\documentclass[12pt]{amsart}

\usepackage{amssymb,amsthm,amstext,latexsym,amsmath}
\usepackage{enumitem}
\usepackage[mathscr]{eucal}
\usepackage{tensor}
\usepackage[pdftex]{color} 
\usepackage{xcolor}
\usepackage{wrapfig} 
\usepackage[all]{xy}
\usepackage{float}
\usepackage[pdftex]{graphicx} 
\usepackage{longtable}
\usepackage{mathtools}
\usepackage{accents}
\usepackage[normalem]{ulem}
\usepackage{cancel}
\usepackage{upgreek}

\input xy
\xyoption{all}

\theoremstyle{plain}
\newtheorem{thm}{Theorem}
\newtheorem{theorem}[thm]{Theorem}
\newtheorem{corollary}[thm]{Corollary}
\newtheorem{lemma}[thm]{Lemma}

\newtheorem*{mr}{Main Result}

\newtheoremstyle{exm}
{9pt}{9pt}{}{}{\bfseries}{}{.5em}{}
\theoremstyle{exm}
\newtheorem{exm}[thm]{Example}
\newtheoremstyle{rmk}
{9pt}{9pt}{}{}{\bfseries}{}{.5em}{}
\theoremstyle{rmk}
\newtheorem{rmk}[thm]{Remark}

\theoremstyle{alg}

\newtheoremstyle{question}
{9pt}{9pt}{}{}{\bfseries}{}{.5em}{}
\theoremstyle{question}

\numberwithin{equation}{section}
\numberwithin{thm}{section}
\numberwithin{figure}{section}

\theoremstyle{definition}
\newtheorem{definition}[thm]{Definition}

\newcommand{\R}{\mathbb{R}}
\newcommand{\C}{\mathbb{C}}
\newcommand{\Z}{\mathbb{Z}}
\newcommand{\sm}{(M,\omega)} 
\newcommand{\comp}{\left(M,\omega,T,\Phi\right)} 

\title[Topological properties of positive complexity one spaces]{On topological properties of positive complexity one spaces}

\author[S. Sabatini]{Silvia Sabatini}
\address{Department Mathematik/Informatik, Universit\"at zu K\"oln,
  Weyertal 86-90, 50931 K\"oln, Germany.}
\email{sabatini@math.uni-koeln.de}

\author[D. Sepe]{Daniele Sepe}
\address{Instituto de Matem\'atica e Estat\'istica, Departamento de
  Matem\'atica Aplicada (GMA), Universidade Federal Fluminense, Campus
  Gragoat\'a, Rua Prof. Marcos Waldemar de Freitas Reis, s/n, S\~ao
  Domingos, Niter\'oi, RJ, 24210--201, Brazil.}
\email{danielesepe@id.uff.br}

\date{\today}

\begin{document}

\begin{abstract}
Motivated by work of Fine and Panov, and of Lindsay and Panov, we
prove that every closed symplectic complexity one space that is positive (e.g. positive
monotone) enjoys topological properties that
Fano varieties with a complexity one holomorphic torus action possess. In particular, such spaces are simply connected, have
Todd genus equal to one and vanishing odd Betti numbers.
\end{abstract}

\maketitle
\section{Introduction}\label{sec:introduction}
A driving (meta-)question in symplectic topology is to understand how closed symplectic manifolds differ from 
smooth complex projective varieties. While
there are examples of closed symplectic manifolds that cannot
be K\"ahler (see, for instance,
\cite{thurston,rez}), it makes
sense to consider refinements of the above problem. Largely inspired
by work of Fine and Panov \cite{fp_hyp,
  fp} and of Lindsay and Panov \cite{lp}, in this paper we prove that a class of
symplectic manifolds with `sufficiently large' torus symmetries share
topological properties with their complex projective counterparts 
(see the Main Result below).

First, we introduce the class of symplectic manifolds that we
consider. To this end, given
a symplectic manifold $\sm$, we denote its first
Chern class by $c_1$.

\begin{definition}\label{defn:conditions}
  A closed symplectic manifold $\sm$ is 
  \begin{itemize}[leftmargin=*]
  \item {\em positive monotone} if there exists $\lambda > 0$ such
    that $c_1 = \lambda[\omega]$, and
  \item {\em symplectic Fano} if there exists a compatible
    almost-complex structure $J$ such that $c_1[A] > 0$ for all non-zero $A \in
    H_2(M)$ that can be represented by a $J$-holomorphic curve.
  \end{itemize}
\end{definition}

\begin{rmk}\label{rmk:convention}
  In some works in the literature, what we call `positive monotone'
  is referred to as `monotone' (sometimes imposing $\lambda =1$ in Definition
  \ref{rmk:convention}), or `symplectic Fano' (see, for instance,
  \cite{lp}). The above definition of symplectic Fano is taken from
  \cite[Remark 11.1.1]{mcduff_sal}. 
\end{rmk}

\begin{rmk}\label{rmk:pos_mon}
  Observe that,
  given an almost complex structure $J$ compatible with $\omega$,
  $J$-holomorphic curves in $(M,\omega)$ are necessarily symplectic.
Hence positive monotone implies symplectic Fano in Definition
  \ref{defn:conditions}.
\end{rmk}

The class of manifolds introduced in Definition
\ref{defn:conditions} can be thought of as the symplectic analog of
smooth {\em Fano} complex varieties, namely those having an ample
anticanonical bundle. In fact, a Fano variety $Y$, together
with the symplectic form induced by pulling back the Fubini-Study form
on projective space along the embedding given by ampleness of the
anticanonical bundle, is necessarily positive monotone. Fano varieties have been extensively studied in
differential, symplectic and algebraic geometry. For the purposes of
this paper, we remark that they are 
are simply
connected (see \cite[Corollary 6.2.18]{isko_prok} and \cite[Remark
3.12]{lindsay}), and have Todd genus equal to one (see \cite[Section
1.8]{hirze} for a definition).

Next we introduce the symmetries that we allow. Throughout this paper, we denote a compact torus by $T$. Moreover, all actions are assumed to be
effective, unless otherwise stated. On a positive monotone/symplectic Fano closed
symplectic manifold $\sm$ we consider {\em Hamiltonian} $T$-actions,
i.e., those for which there exists a $T$-invariant smooth map $\Phi : M \to
\left(\mathrm{Lie}(T)\right)^*$, called {\em moment map}, such that, for all $\xi \in
\mathrm{Lie}(T)$, $ \iota_{X_{\xi}}\omega = d \langle \Phi, \xi \rangle$,
where $\mathrm{Lie}(T)$ denotes the Lie algebra of $T$, $X_{\xi} \in \mathfrak{X}(M)$ is the vector field induced by
$\xi$, and $\langle \cdot , \cdot \rangle$ is the natural pairing between
$\left(\mathrm{Lie}(T)\right)^*$ and $\mathrm{Lie}(T)$. A {\em Hamiltonian $T$-space} is a symplectic manifold $\sm$ endowed
with an effective Hamiltonian $T$-action. Such a space, together with a
choice of moment map $\Phi$, is denoted by
$\comp$. To make sense of when torus symmetries are `sufficiently
large', we introduce the following notion. 

\begin{definition}\label{defn:complexity}
  The {\em complexity} of a Hamiltonian $T$-space $\comp$
  is $\frac{1}{2}\dim M - \dim T$.
\end{definition}

Intuitively, the smaller the complexity, the larger the
symmetry. Moreover, a simple symplectic argument shows that
the complexity of a Hamiltonian $T$-space is always
non-negative. Henceforth, a Hamiltonian $T$-space of complexity $k$ is
simply referred to as a {\em complexity $k$ space}. Complexity zero spaces are known as symplectic
toric manifolds and it is known that positive monotone complexity zero
spaces are $T$-equivariantly symplectomorphic to toric Fano varieties,
i.e., Fano varieties $Y$ endowed with a holomorphic
$T_{\C}$-action, where $T_{\C} = T \otimes_{\R} \C$ and $\dim_{\C} Y =
\dim_{\C} T_{\C}$. 

This paper begins the study of the relation between positive
monotone (respectively symplectic Fano) complexity one spaces and Fano
varieties $Y$ equipped with an effective holomorphic 
$T_{\C}$-action satisfying $\dim_{\C} Y =
\dim_{\C} T_{\C} + 1$. To the best
of our knowledge, this is an unexplored problem except for low real
dimensions, namely 2 and 4. While the two-dimensional case is not
particularly interesting, a positive monotone 4-dimensional
complexity one space is $S^1$-equivariant
symplectomorphic to a del Pezzo surface (i.e., a Fano surface)
endowed with a holomorphic $\C^*$-action. This can be proved using
techniques that underpin
the classification of closed Hamiltonian $S^1$-spaces in dimension 4 (see
\cite{karshon}). 

The main result of this paper is the following:

\begin{mr}
  If $\comp$ is a closed complexity one space that is either positive monotone, or
  symplectic Fano with respect to a compatible $T$-invariant
  almost-complex structure, then $M$ is simply connected, its odd
  Betti numbers vanish, and $\sm$ has Todd genus
  equal to 1.
\end{mr}

\begin{rmk}\label{rmk:on_main_result}
  As mentioned above, Fano varieties are necessarily simply connected
  and have Todd genus equal to 1. It can be checked that a Fano variety $Y$ that is
  endowed with an effective holomorphic 
  $T_{\C}$-action satisfying $\dim_{\C} Y =
  \dim_{\C} T_{\C} + 1$ has vanishing odd Betti numbers.
\end{rmk}

The above result is very much inspired by work of Fine and Panov \cite{fp_hyp,
  fp} and of Lindsay and Panov \cite{lp}, and should be placed in the
context of the broader question of studying the relation between
closed positive monotone symplectic manifolds and Fano varieties in
the presence of a torus action.
Without assuming the existence of a Hamiltonian torus action, in real dimension $4$ every closed positive monotone symplectic manifold is diffeomorphic to
a del Pezzo surface, i.e., a Fano two-fold (see \cite{mcduff_structure,gromov,taubes}). However, this need not hold
in higher dimensions (see \cite{fp_hyp,rez} for a
counterexample). However, in \cite[Section 1.2 and 7]{fp}
it is conjectured that a closed positive monotone manifold of real dimension 6 with a non-trivial
Hamiltonian $S^1$-action must be diffeomorphic to a Fano
threefold.
In \cite{lp} Lindsay and Panov make some important steps towards
proving the above conjecture as they show that,
like Fano varieties, such a symplectic manifold is simply connected and its Todd genus is
one. Moreover, under various additional hypotheses either on the topology of the manifold or on the type of the action, there is evidence
that real six-dimensional positive monotone symplectic manifolds with a Hamiltonian
$S^1$-action are either diffeomorphic or $S^1$-equivariantly
symplectomorphic to Fano three-folds endowed with a holomorphic $\C^*$-action (see
\cite{cho,cho2,cho3,fp,gvhs,lp,mcduff_six,tolman}). \\

Our Main Result should be compared with (some of) the main results in
\cite{lp,lindsay}. If, on the one hand, the hypotheses in \cite{lp} are
weaker than those of our Main Result, in that they deal with
complexity two spaces, the results therein are specific
to the real 6-dimensional
case, whereas our result applies to all dimensions. Moreover, while we are
able to conclude that in our case the odd Betti numbers vanish, the
corresponding statement in the real 6-dimensional case with a
Hamiltonian $S^1$-action only
holds by imposing further mild conditions (see \cite[Theorem
14.4]{lindsay}). (In fact, there exist complexity one Fano 3-folds with $b_3 > 0$, see \cite[Example
14.8]{lindsay}.) Finally, it is important to remark that the
techniques in \cite{lp} are significantly more sophisticated than
those used in this paper, seeing as, for instance, \cite{lp} uses
Seiberg-Witten theory (cf. Section \ref{sec:proof-main-result}
below). \\

The proof of the above result comes from combining several well-known properties
of closed complexity one spaces under the assumption that the space be
`positive' (see Definition \ref{defn:positive} for details). Closed
complexity one spaces as in the hypothesis of our Main Result satisfy
this positivity condition (see Lemma \ref{thm:main}). Our strategy is
simple: we prove that the assumption of positivity on a closed
complexity one space limits the topology of the connected components
of the fixed point set of the action (see Theorem \ref{thm:main}) and
this is the key ingredient in the proof of our Main Result. To prove
Theorem \ref{thm:main} we use the Duistermaat-Heckman function, the
fact that its minimum need be attained at a vertex for closed
complexity one spaces (see Theorem \ref{thm:log-concavity} and
Corollary \ref{cor:min_DH}), and a topological restriction in the case in
which the vertex that attains the minimum of this function is the
image of a 2-dimensional component of the fixed point set (see Lemma
\ref{lemma:min_fixed_surface_negative}). \\

The paper is structured as follows. In Section
\ref{sec:basic-prop-hamilt} we recall the basics of (closed)
Hamiltonian $T$-spaces. Section
\ref{sec:some-prop-clos} deals with (closed) complexity one spaces and
its aim is to prove Lemma
\ref{lemma:min_fixed_surface_negative}. While most results contained
therein are standard, there are a few
observations that we could not find elsewhere in the literature,
including Lemma \ref{lemma:min_fixed_surface_negative}. Most (if not
all) of the material in Sections \ref{sec:basic-prop-hamilt} and
\ref{sec:some-prop-clos} are probably well-known to experts, and it is
included for completeness. The notion of positivity as well as the
proof of our Main Result can be found in Section \ref{sec:proof-main-result}. 

\subsection*{Acknowledgments} The authors were partially supported by
SFB-TRR 191 grant {\it Symplectic Structures in Geometry, Algebra and
  Dynamics} funded by the Deutsche Forschungsgemeinschaft. D.S. was
partially supported by CNPq grant {\it Bolsa de Produtividade em
  Pesquisa} 3058/2015-0. This study was financed in part by the
Coordena\c{c}\~ao de Aperfei\c{c}oamento de Pessoal de N\'ivel
Superior -- Brasil (CAPES) -- Finance code 001. D.S. would like to
thank Universit\"at zu K\"oln for the kind hospitality during the
period in which this project came to being.

\section{Basic properties of (closed) Hamiltonian $T$-spaces}\label{sec:basic-prop-hamilt}
Throughout this section, given a compact torus $T$, its Lie algebra and the
lattice therein are denoted by
$\mathfrak{t}$ and $\ell =
\ker\left( \exp: \mathfrak{t} \to T\right)$ respectively. The aim of
this section is to recall a few fundamental facts about (closed)
Hamiltonian $T$-spaces, i.e., symplectic manifolds endowed with an effective
Hamiltonian $T$-action. \\

First, we recall the local normal form near fixed points of the
$T$-action, which is a special case of a more general result due to
Marle, and Guillemin and Sternberg (see \cite{marle,gs-local}). Given
a Hamiltonian $T$-space $\comp$, the set of fixed points of the
action is denoted by $M^T$ and endowed with the subspace topology. Fix
a Hamiltonian $T$-space $\comp$, set $\dim M = 2n$, fix $p \in M^T$, and a $T$-invariant compatible almost
complex structure $J$. Since $p$ is a fixed point, there is a
$\C$-linear $T$-action on $T_pM$ that is isomorphic to a $T$-action on $\C^n$ determined by an injective
homomorphism $\rho : T \to (S^1)^n$, where
$(S^1)^n \subset \mathrm{GL}(n;\C)$ is the subgroup of diagonal
matrices whose entries have norm one. The homomorphism $\rho$ is known
as the {\em (symplectic) slice representation}. 

\begin{definition}\label{defn:iso_weights}
  Let $\comp$ be a Hamiltonian $T$-space and let $p$ be a fixed
  point. The differential at the identity
  of the components of the slice representation determines elements $\alpha_1,\ldots,\alpha_n
  \in \mathfrak{l}^*$, called the {\em isotropy weights} of the
  $T$-action at $p$.
\end{definition}

\begin{rmk}\label{rmk:iso_weights}
  \mbox{}
  \begin{itemize}[leftmargin=*]
  \item Since the slice
    representation is injective, it follows that the $\Z$-span of the isotropy weights
    at a fixed point is $\ell^*$.
  \item Since both $T$ and $(S^1)^n$ are
    connected, the isotropy weights $\alpha_1,\ldots,\alpha_n$ (up to
    permutation) determine the slice representation $\rho$.
  \item Let $F \subset M^T$ be a connected component. Then for any $p, p' \in F$
    the isotropy weights of the $T$-action at $p$ and $p'$ are
    equal. Thus it makes sense to talk about the isotropy weights of the
    $T$-action at $F$.  
  \end{itemize}
\end{rmk}

Suppose that $\alpha_1,\ldots,\alpha_n
  \in \mathfrak{l}^*$ are the isotropy weights at $p$. Endowing $\C^n$ with the standard symplectic structure
$\omega_{\mathrm{st}}$, the linear
$T$-action on $\C^n$, determined by $\alpha_1,\ldots,\alpha_n$ as above,
is Hamiltonian and one of the moment maps is given by 
\begin{equation}
  \label{eq:4}
  \Phi_{\mathrm{lin}}\left(z_1,\ldots,z_n\right) = \frac{1}{2} \sum\limits_{j=1}^n
  \alpha_j |z_j|^2 + \Phi(p).
\end{equation}
\noindent
We henceforth refer to the above Hamiltonian $T$-space as the {\em
  linear model} at the fixed point $p$. The following result is a (very!) particular case of the local normal form for
Hamiltonian group actions by compact Lie groups due to Marle,
Guillemin and Sternberg (see \cite{gs-local,marle}).

\begin{thm}[Local normal form at fixed points]\label{thm:lnf}
  Let $\comp$ be a Hamiltonian $T$-space and let $p \in M^T$ be a
  fixed point. Then
  there exist $T$-invariant open neighborhoods $U \subset M$ and $V
  \subset \C^n$ of $p$
  and $0$ respectively, and a $T$-equivariant symplectomorphism 
  $\Psi: (U,\omega)\to (V,\omega_{\mathrm{st}})$ such that $\Phi =
  \Phi_{\mathrm{lin}} \circ \Psi$, where $\Phi_{\mathrm{lin}}$ is as
  in \eqref{eq:4}.
\end{thm}

Next we state without proof the following basic, yet important, result (see \cite[Lemma 5.53]{mcduff}).

\begin{lemma}\label{lemma:properties}
  Let $\comp$ be a Hamiltonian $T$-space. For any closed $H \subset
  T$, each connected components of the set of points that are
  fixed by $H$ is a symplectic submanifold of
  $\sm$.
\end{lemma}

We conclude this section by recalling two results (without proof) concerning {\em closed} Hamiltonian
$T$-spaces. The first one is the well-known milestone due to Atiyah, Guillemin and Sternberg (see
\cite{atiyah,gs}). 

\begin{thm}\label{thm:conn_conv}
  Let $\comp$ be a closed Hamiltonian $T$-space. Then the fibers of $\Phi$
  are connected and $\Phi(M) \subset \mathfrak{t}^*$ is the convex
  hull of the image of the connected components of $M^T$.
\end{thm}

\begin{rmk}\label{rmk:poly}
  Observe that, under the hypotheses of Theorem \ref{thm:conn_conv}, $M^T$ has finitely many connected
  components. Thus the moment map image of a closed Hamiltonian $T$-space
  is a convex compact polytope in $\mathfrak{t}^*$. 
\end{rmk}

The last result is a special case of a theorem of Li
(see \cite{li}).

\begin{theorem}\label{thm:fund_group}
  Let $\comp$ be a closed Hamiltonian $T$-space. For any $\alpha \in
  \Phi(M)$, $\pi_1(M) \cong \pi_1(M_\alpha)$, where $M_\alpha= \Phi^{-1}(\alpha)/T$ denotes
  the reduced space at $\alpha$. 
\end{theorem}

\section{Some properties of (closed) complexity one spaces}\label{sec:some-prop-clos}
In this section, we specialize to (closed) 
complexity one spaces and
we prove results that are needed in the proofs of Section
\ref{sec:proof-main-result}. Throughout this section, $\comp$ denotes
a complexity one space unless otherwise stated. 

\subsection{Fixed surfaces}\label{sec:fixed-surfaces}
We begin by showing that the complexity of 
$\comp$ being one restricts the possible dimensions of the connected
components of $M^T$.

\begin{corollary}\label{cor:connected}
  Let $\comp$ be a complexity one space. The connected components
  of $M^T$ are either points or symplectic surfaces.
\end{corollary}

\begin{proof}
  Let $F \subset M^T$ be a connected component and fix $p \in F$. By Lemma
  \ref{lemma:properties}, we know that $F$ is a symplectic submanifold
  so we only have to bound its dimension. By
  Theorem \ref{thm:lnf}, it suffices to consider the linear model,
  i.e., $p
  = 0 \in \C^n$ and the $T$-action is given by an injective
  representation $\rho: T \to (S^1)^n$. Observe that $T_p F =
\bigcap\limits_{t \in T} \ker (\mathrm{id} - \rho(t))$ (see
\cite[proof of Lemma 5.53]{mcduff}). Since $\rho$ is injective and $n
= \dim T +1$, it
follows that $\dim T_pF$ equals either $0$ or $2$. 
\end{proof}

Next we deduce properties of the linear model at (and, hence, of the
local behavior of the $T$-action near) a fixed point lying on a
2-dimensional connected component of $M^T$. Such a connected component
is henceforth referred to as a {\em fixed surface}. Fix such a surface $\Sigma \subset M^T$ and a point $p \in
\Sigma$. As in the proof of Corollary \ref{cor:connected}, assume, without
loss of generality, that $M = \C^n$, $p = 0$ and the $T$-action is given by an injective
representation $\rho: T \to (S^1)^n$. Since $p$ lies on a fixed
surface, it follows that one of the isotropy weights of the action is
0, i.e., the $T$-action fixes a complex line in
$\C^n$. Therefore, there exists an isomorphism $\C^n \cong \C \times
\C^{n-1}$ so that for all $t \in T$ and all $(z,w) \in \C
\times \C^{n-1}$, 
\begin{equation}
  \label{eq:2}
  t \cdot (z,w) = (z, \hat{\rho}(t)w)
\end{equation}
\noindent
for some isomorphism $\hat{\rho}
: T \to (S^1)^{n-1}$. In what follows, we ignore the zero isotropy weight at $p$ and refer
to $\alpha_1,\ldots, \alpha_{n-1} \in \ell^*$ as the isotropy weights
at $p$.

\begin{rmk}\label{rmk:compl_0}
  In the above description, for any $z_0 \in \C$, the subspace
  $\{z_0\} \times \C^{n-1}$ is a symplectic subspace of $\C \times
  \C^{n-1}$ whose induced symplectic form is denoted by $\omega_0$
  (and is symplectomorphic to the standard symplectic form on
  $\C^{n-1}$). Let $\Phi_0$ denote the moment map of the effective
  Hamiltonian $T$-action on $(\{z_0\} \times \C^{n-1} , \omega_0 )$
  given by restricting the $T$-action on $\C^n$. The complexity of the
  Hamiltonian $T$-space $\left(\{z_0\} \times
    \C^{n-1}, \omega_0, \Phi_0\right)$ is zero, i.e., $T$ acts on
  $(\{z_0\} \times \C^{n-1} , \omega_0 )$ in a toric fashion.
\end{rmk}

In the following simple result, we use the above discussion in the
case in which $\comp$ is a closed complexity one space.

\begin{lemma}\label{lemma:fixed_surfaces}
  Let $\comp$ be a closed complexity one space and let $\Sigma \subset
  M^T$ be a fixed surface. Then $\Sigma$ is the preimage of a vertex
  of $\Phi(M)$. 
\end{lemma}
\begin{proof}
  Compactness of $M$ implies compactness of $\Sigma$ as it is 
  closed. Fix $p \in \Sigma$. Theorem \ref{thm:lnf} and the above discussion imply
  that there exists an open neighborhood of $U_p$ of $p$ such that
  $\Phi(U)$ is the image of an open neighborhood $V_p$ of $ 0 \in \C^n \cong
  \C \times \C^{n-1}$ under the map 
  \begin{equation}
    \label{eq:6}
    \Phi_{\mathrm{lin}}(z,w_1,\ldots,w_{n-1}) = \frac{1}{2} \sum\limits_{j=1}^{n-1}
    \alpha_j |w_j|^2 + \Phi(p), 
  \end{equation}
  \noindent
  (cf. formula \eqref{eq:4}). Since the isotropy weights at a fixed
  surface are well-defined (see Remark \ref{rmk:iso_weights}), the
  above statement holds for all $p \in \Sigma$. Since $\Sigma$ is
  compact, there exists finitely many $p_1,\ldots, p_N$ such that
  $\Sigma$ is covered by $U_{p_1},\ldots, U_{p_N}$. Set $U :=
  \bigcup\limits_{l=1}^N U_{p_l} \subset M$ and $V:=
  \bigcap\limits_{l=1}^N V_{p_l} \subset \C^n$. Observe that $V$ is a
  $T$-invariant open neighborhood of $\C^n \cong \C \times \C^{n-1}$
  and that $\Phi_{\mathrm{lin}}(V)$ is an open subset of
  $\Phi(\Sigma) + \R_{\geq 0 } \langle
  \alpha_1,\ldots,\alpha_{n-1} \rangle$. Moreover, by construction,
  $\Phi_{\mathrm{lin}}(V) \subset \Phi(U)$. Since $\Phi$ is proper, it
  follows that $U$ contains an open neighborhood of $\Sigma$ saturated
  by the fibers of $\Phi$ which are connected.
  Hence, possibly changing $U$ with this open neighborhood $U'$ saturated by $\Phi$, the
  vertex of $\Phi(U')$, whose preimage is $\Sigma$, must be a vertex of $\Phi(M)$.
\end{proof}

\begin{rmk}\label{rmk:higher_comp}
  Lemma \ref{lemma:fixed_surfaces} can be generalized to closed
  complexity $k$ spaces
  by substituting fixed surfaces with connected components of the
  fixed point set of dimension $2k$. However, in complexity $k \geq 2$, it is not true that
  2-dimensional connected components of the fixed point set are either
  level sets of the moment map or that they are contained only in the
  preimages of vertices of the moment map image.
\end{rmk}

Using Lemma \ref{lemma:fixed_surfaces}, given a closed complexity one
space $\comp$, we characterize the
preimage of an edge of $\Phi(M)$ that is incident to a vertex whose
preimage is a fixed surface. More precisely, the following result
holds.

\begin{lemma}\label{lemma:edges}
  Let $\comp$ be a closed complexity one space and suppose that a
  vertex $v \in \Phi(M)$ is the image of a fixed surface. Then the
  preimage of any closed edge incident to $v$ is a closed
  4-dimensional symplectic submanifold of $\sm$ endowed with an 
  effective Hamiltonian $S^1$-action. 
\end{lemma}

\begin{proof}
  Set $\Sigma := \Phi^{-1}(v)$. The proof of Lemma
  \ref{lemma:fixed_surfaces} shows that the image of $\Phi$ near $v$
  is contained in $\Phi(\Sigma) + \R_{\geq 0 } \langle
  \alpha_1,\ldots,\alpha_{n-1} \rangle$, where
  $\alpha_1,\ldots,\alpha_{n-1}$ are the isotropy weights at the fixed surface
  $\Sigma$. Let $e_1,\ldots, e_{n-1}$ denote the closed edges incident
  to $v$ so that, for all $i=1,\ldots, n-1$, $e_i \subset \Phi(\Sigma)
  + \R_{\geq 0} \alpha_i$. Fix $i=1,\ldots, n-1$ and a point $p \in
  \Sigma$. Using the local normal form of Theorem \ref{thm:lnf},
  together with the discussion preceding Lemma
  \ref{lemma:fixed_surfaces}, we may identify an open neighborhood $U$
  of $p$ with an open neighborhood of $0 \in \C^n \cong \C \times
  \C^{n-1}$ with $T$-action and moment map  given as in \eqref{eq:2}
  and \eqref{eq:6} respectively. Under this identification,
  $\Phi^{-1}(e_i) \cap U$ is given by the 4-dimensional subspace 
  $$ \left\{ (z,w_1,\ldots,w_{n-1}) \mid w_1 = 0, \ldots, w_{i-1} = 0,
    w_{i+1} = 0, \ldots, w_{n-1} = 0 \right\}.$$
  \noindent
  Moreover, setting $\mathfrak{h}_i:= \ker \alpha_i$ and $H_i := \exp
  \left(\mathfrak{h}_i\right)$, we have that $q \in \Phi^{-1}(e_i)
  \cap U$ if and only if it is fixed by $H_i$. (Note that the weights
  $\alpha_1,\ldots,\alpha_{n-1}$ are linearly independent since their
  $\Z$-span must be $\ell^*$, as otherwise the effectiveness of the $T$-action
  would be contradicted). Since $\Phi^{-1}(e_i)$
  is connected (as the fibers of $\Phi$, and $e_i$, are), Lemma
  \ref{lemma:properties} implies that $\Phi^{-1}(e_i)=: M_i$ is a
  4-dimensional symplectic submanifold of $\sm$ whose points are fixed
  by $H_i$. Moreover, $M_i$ is closed as it
  is the preimage of a compact subset under the proper map $\Phi$. Set $\omega_i :=
  \omega|_{M_i}$. It remains to show that $\left(M_i,\omega_i\right)$
  is endowed with an effective Hamiltonian $S^1$-action. Since $\alpha_1,\ldots,\alpha_{n-1} \in \ell^*$ are
  a basis (see Remarks \ref{rmk:iso_weights} and \ref{rmk:compl_0}),
  consider the dual basis $a_1,\ldots, a_{n-1} \in \ell$. By
  construction, $\exp\left(\langle a_i \rangle\right) \subset T$ is isomorphic
  to $S^1$ and it acts in a Hamiltonian fashion on
  $\left(M_i,\omega_i\right)$. To check that this action is effective,
  observe that, by construction, it is effective locally near $p$;
  this can be checked directly in the linear model at $p$.
\end{proof}

We conclude our discussion of fixed surfaces of closed complexity one
spaces with a simple observation regarding the case in which there
is one with positive genus.

\begin{lemma}\label{lemma:positive_genus}
  Let $\comp$ be a closed complexity one space. If there exists a
  fixed surface $\Sigma_0 \subset M^T$ whose genus $g(\Sigma_0)$ is
  positive, then, for any vertex $v \in \Phi(M)$, $\Phi^{-1}(v)$ is a
  fixed surface of genus $g(\Sigma_0)$.
\end{lemma}

\begin{proof}
  By Lemma \ref{lemma:fixed_surfaces}, $\Phi(\Sigma_0)$ is a vertex of
  $\Phi(M)$, say $v_0$. Then, by Theorem \ref{thm:fund_group},
  $\pi_1(M) \cong \pi_1(M_{v_0}) \cong \pi_1(\Sigma_0)$, where
  $M_{v_0}$ is the reduced space at $v_0$. By assumption, $\pi_1(\Sigma_0)$ is not trivial. Since the first isomorphism above holds for
  any vertex $v$ of $\Phi(M)$, it follows that the preimage of any
  other vertex is a fixed surface whose genus equals that of
  $\Sigma_0$ as desired.
\end{proof}

\subsection{The Duistermaat-Heckman function and its minimum}\label{sec:duist-heckm-funct}
Let $DH : \Phi(M) \to
\R$ denote the Duistermaat-Heckman function associated to a closed
complexity one space $\comp$,
namely $DH(\alpha)$ is the symplectic volume of the reduced space at
$\alpha$ (see \cite{dh}). First, we state the following result due to
Cho and Kim without proof (see \cite{ck}).

\begin{thm}\label{thm:log-concavity}
  The Duistermaat-Heckman function of a closed complexity one space
  is log-concave, i.e., $\log DH$ is a concave function.
\end{thm}

Combining Theorems \ref{thm:conn_conv} and \ref{thm:log-concavity}, we
obtain the following:

\begin{corollary}\label{cor:min_DH}
  The minimum of the Duistermaat-Heckman function of a closed
  complexity one space is attained at a vertex of the moment map image.
\end{corollary}
\begin{proof}
  Let $\comp$ be a closed complexity one space and let $DH :
  \Phi(M) \to \R$ denote its Duistermaat-Heckman function. Theorem
  \ref{thm:log-concavity} asserts that $\log DH$ is concave. Thus, to
  prove the result, it suffices to show that $\log DH$ attains its
  minimum at a vertex of $\Phi(M)$. This follows at once by convexity
  of $\Phi(M)$ (see Theorem \ref{thm:conn_conv}), since a concave
  function on a compact convex polytope must attain its minimum at a vertex. 
\end{proof}

The next result 
describes a topological restriction on a fixed surface whose image
corresponds to the minimum of the Duistermaat-Heckman function (see
Lemma \ref{lemma:fixed_surfaces} and Corollary \ref{cor:min_DH}).

\begin{lemma}\label{lemma:min_fixed_surface_negative}
  Let $\comp$ be a closed complexity one space and let $v \in
  \Phi(M)$ be a vertex that attains the minimum of the
  Duistermaat-Heckman function $DH$. If $\Phi^{-1}(v)$ is a fixed surface, then
  $$c_1(N)[\Phi^{-1}(v)] \leq 0, $$
  \noindent
  where $N$ and $c_1(N)$ denote the normal bundle
  to $\Phi^{-1}(v)$ and its first Chern class respectively.
\end{lemma}

\begin{proof}
  Set $\Sigma:= \Phi^{-1}(v)$ and fix a $T$-invariant almost complex structure $J$. Let
  $\alpha_1,\ldots, \alpha_{n-1} \in \ell^*$ be the isotropy weights
  of $\Sigma$ (see Definition \ref{defn:iso_weights} and Remark \ref{rmk:iso_weights}).
  Since $\alpha_1,\ldots, \alpha_{n-1} $ form a basis of $\ell^*$ (see
  Remarks \ref{rmk:iso_weights} and \ref{rmk:compl_0}), 
  the normal bundle $N$ to
  $\Sigma$ splits as the sum of $n-1$ complex 
  line bundles $L_1,\ldots, L_{n-1}$, each $L_i$ corresponding to
  exactly one
  $\alpha_i \in \ell^*$, for $i=1,\ldots,n-1$. By additivity of the
  first Chern class, the result is proved if we show that, for all
  $i=1,\ldots, n-1$, $c_1(L_i)[\Sigma] \leq 0$. 

  Let $e_1,\ldots,e_{n-1}$
  denote the (closed) edges of $\Phi(M)$ incident to $v$ so that, for
  all $i=1,\ldots, n-1$, $e_i \subset v + \R_{\geq 0}\alpha_i$. For any $i=1,\ldots,
  n-1$, $L_i$ is the normal bundle of $\Sigma$ inside the closed
  $4$-dimensional symplectic submanifold $(M_i,\omega_i)$ of $M$ given
  by $\Phi^{-1}(e_i)$ and $\omega_i = \omega|_{M_i}$. By Lemma
  \ref{lemma:edges}, for each $i=1,\ldots,n-1$, there exists an
  injective homomorphism
  $\chi_i: S^1
  \to T$ such that $\Phi_i:=\chi_i^* \circ \Phi$ is the moment map of a Hamiltonian $S^1$-action on $(M_i,\omega_i)$, where
  $\chi_i^* : \mathrm{Lie}(T)^* \to \R \cong \mathrm{Lie}(S^1)^*$ is the
  homomorphism induced by $\chi_i$. Fix $i=1,\ldots, n-1$ and,
  without loss of generality, suppose that 
  $\Phi_i(\Sigma) = 0$ and $\Phi_i(M_i) \subset \R_{\geq
    0}$. For all $t >0$ sufficiently small, the reduced space
  $\Phi_i^{-1}(t)/S^1$ is symplectomorphic to the reduced space
  $\Phi^{-1}(v + t \alpha_i)/T$, since, by construction,
  $\Phi^{-1}(v+t \alpha_i) = \Phi^{-1}_i(t)$. Since $v$ is assumed to be a minimum of the Duistermaat-Heckman
  function for $\comp$, it follows that the Duistermaat-Heckman
  function $DH_i$ for $(M_i,\omega_i,\Phi_i)$ is a
  non-decreasing function in the interval $(0,t_0)$, for $t_0 > 0$
  sufficiently small. Using \cite[Lemma
  2.12]{karshon}, it follows that $c_1(L_i)[\Sigma] \leq 0$. Since $i
  =1,\ldots,n-1$ is arbitrary, this proves the desired result.
\end{proof}

\section{Positivity and the proof of our main result}\label{sec:proof-main-result}
The aim of this section is to provide a proof of our Main Result by
showing that its conclusions hold if the closed complexity one space
is assumed to be `positive' in the following sense.

\begin{definition}\label{defn:positive}
  A closed complexity one space $\comp$ is {\em positive} if,
  for any fixed surface $\Sigma \subset M^T$, $c_1[\Sigma] >0$, where
  $c_1$ is the first Chern class of $\sm$.
\end{definition}

\begin{exm}\label{exm:positive}
  If all the connected components of $M^T$ of a closed
  complexity one space $\comp$ are isolated fixed points then $\comp$
  is positive. This can be used to show that positive does not imply positive monotone or symplectic
  Fano in the sense of Definition \ref{defn:conditions}. For instance,
  it is not hard to construct a closed complexity one space of
  dimension four all of whose fixed points are isolated that has a
  symplectic sphere of self-intersection equal to -2, and which is
  $J$-holomorphic with respect to an $S^1$-invariant compatible almost
  complex structure. The existence of
  such a sphere prevents the symplectic manifold from being positive
  monotone or symplectic Fano.
\end{exm}

The next result illustrates why we introduce positivity.

\begin{lemma}\label{lemma:positive_ok}
  A closed complexity one space that is either positive monotone, or symplectic Fano with respect to a compatible $T$-invariant
  almost-complex structure, is positive. 
\end{lemma}

\begin{proof}
  Since fixed surfaces are symplectic submanifolds (see Lemma \ref{lemma:properties}), it follows
  that a positive monotone closed complexity one space is
  positive. On the other hand, suppose that a closed complexity one space is symplectic Fano with respect
  to a $T$-invariant compatible almost complex structure. By 
  \cite[Proof of Lemma 5.53]{mcduff} it follows that any fixed surface of $\comp$
  is $J$-invariant and, therefore, $J$-holomorphic as it is two-dimensional. 
\end{proof}

Using Lemma \ref{lemma:positive_ok}, our Main Result is a simple
consequence of the following:

\begin{thm}\label{thm:intro}
  Let $\comp$ be a positive closed complexity
  one space. Then $M$ is simply connected, its odd Betti numbers vanish and the Todd
  genus of $\sm$ equals one.
\end{thm}

Thus it remains to prove Theorem \ref{thm:intro}. To this end, we
first prove the following:

\begin{thm}\label{thm:main}
  Let $\comp$ be a positive closed complexity
  one space. The connected components
  of $M^T$ are either points or spheres.
\end{thm}

\begin{rmk}\label{rmk:tall}
  By Corollary \ref{cor:connected}, the statement of Theorem \ref{thm:main} is equivalent to the
  following: \\

  \noindent
  {\em The connected components of the fixed point set of a
    positive closed complexity one space are simply connected.} 
\end{rmk}

\begin{proof}[Proof of Theorem \ref{thm:main}]
  Suppose that the statement does not hold. By Corollary
  \ref{cor:connected}, $M^T$ contains a fixed surface of positive
  genus $g$. By Lemma \ref{lemma:positive_genus}, 
  the preimage of any vertex of $\Phi(M)$ is a fixed surface of
  genus $g$. Fix a vertex $v \in \Phi(M)$ and let $\Sigma \subset M^T$ be
  its preimage under $\Phi$. If $N$ denotes the normal bundle to $\Sigma$, then the
  first Chern class of $N$ must satisfy $c_1(N)[\Sigma] > 0$, for 
  \begin{equation}\label{ineq c1}
     c_1(N) [\Sigma] = c_1[\Sigma] - c_1(\Sigma)[\Sigma] = c_1[\Sigma]
     +2g - 2 > 0, 
   \end{equation}
  \noindent
  where $c_1(\Sigma)$ is the first Chern class of the tangent bundle
  to $\Sigma$, and the inequality follows by positivity of $\comp$ and
  $g > 0$. 
  
  To derive a contradiction we use the Duistermaat-Heckman function
  $DH$. By Corollary \ref{cor:min_DH}, the minimum of $DH$ is attained
  at a vertex $m$ of $\Phi(M)$. However, this is impossible by Lemma \ref{lemma:min_fixed_surface_negative}. 
\end{proof}

\begin{rmk}\label{rmk:taller}
  Following \cite{kt2,kt3}, we say that a closed complexity one space is {\em
    tall} if all its reduced spaces are two-dimensional topological
  manifolds. If a closed complexity one space $\comp$ is not tall,
  \cite[Corollary 2.4]{kt3} states that the set of $\alpha \in
  \Phi(M)$ such that $M_{\alpha}= \Phi^{-1}(\alpha)/T$ is a point is the union of closed
  faces. Take one such closed face $\Delta \subset \Phi(M)$; since it
  is closed, it contains a vertex, say $\alpha$. Arguing as above and
  using Theorem \ref{thm:fund_group}, we
  have that the fundamental group of any reduced space is trivial and,
  therefore, all connected components of $M^T$
  are simply connected (see Lemma \ref{lemma:fixed_surfaces}). Therefore, Theorem \ref{thm:main} holds
  without the positivity assumption if $\comp$ is not tall. As such,
  Theorem \ref{thm:main} should be seen as a statement about 
  tall closed complexity one spaces (classified
  in \cite{kt1,kt2,kt3}) under the assumption of positivity in the
  sense of Definition \ref{defn:positive}.
\end{rmk}

We conclude this section by proving Theorem \ref{thm:intro} and,
consequently, our Main Result.

\begin{proof}[Proof of Theorem \ref{thm:intro}]
  Let $\comp$ be a positive closed complexity
  one space. By Theorem \ref{thm:main}, the connected components of
  $M^T$ are simply connected. Therefore, arguing as above, we have
  that the reduced space at any vertex of $\Phi(M)$ is simply
  connected. Using Theorem \ref{thm:fund_group}, simple connectedness of $M$ follows. 

  To prove the remaining statements, choose a
  generic $S^1 \subset T$ with the property that $M^{S^1} = M^T$ (this
  can be done because $M$ is compact), and let $F_1,\ldots, F_N$ be the
  connected components of $M^{S^1} = M^T$. To prove that the Todd genus
  equals one, we use the same techniques as in \cite[Corollary
  1.4]{lp}; the argument is included below for completeness. Recall the
  following formula (see \cite[Section 5.7]{hirze}):
  \begin{equation}
    \label{eq:1}
    \chi_y(M) = \sum\limits_{j=1}^N
    (-y)^{d_j}\chi_y(F_j), 
  \end{equation}
  \noindent
  where $\chi_y$ is the Hirzebruch genus and, for all $j=1,\ldots, N$, $d_j$ is the number of
  negative isotropy weights for the $S^1$-action (counted with multiplicity)
  at $F_j$ (see Definition \ref{defn:iso_weights} and Remark \ref{rmk:iso_weights}). Observe that,
  if $F_j$ is the component corresponding to the minimum of the moment map of the 
  $S^1$-action, then $\Phi(F_j)$ is a vertex of $\Phi(M)$. By
  Theorem \ref{thm:main}, $F_j$ is necessarily a sphere or a point and, in
  both cases, the Todd genus is one.
  Evaluating \eqref{eq:1} at $y=0$ and observing that $\chi_0$ is precisely the
  Todd genus (by definition of the generating functions of the
  Hirzebruch and Todd genera, see \cite[Sections 1.8 and 5.4]{hirze}),
  we obtain that the Todd genus of $M$ is 1 as desired. 

  To see that
  the odd Betti numbers vanish, we combine Theorem \ref{thm:main} with the
  well-known formula
  \begin{equation}
    \label{eq:3}
    H^* (M; \R) = \bigoplus\limits_{j=1}^N H^{*-2d_j}(F_j;\R),
  \end{equation}
  \noindent
  where $d_j$ is as above. Formula \eqref{eq:3} is a consequence of the fact that the moment map for the
  $S^1$-action is perfect Morse-Bott (see \cite{kirwan}).
\end{proof}



\begin{thebibliography}{99}

\bibitem{atiyah}
  \newblock M.F. Atiyah,
  \newblock {\em Convexity and commuting {H}amiltonians},
  \newblock Bull. London Math. Soc., {\bf 14}, no. 1, (1982), 1 -- 15.

\bibitem{ck}
  \newblock Y. Cho, M.K. Kim, 
  \newblock {\em Log-concavity of complexity one {H}amiltonian torus
    actions},
  \newblock C. R. Math. Acad. Sci. Paris, {\bf 350}, no. 17-18,
  (2012), 845 -- 848.

\bibitem{cho}
  \newblock Y. Cho,
  \newblock {\em Classification of six dimensional monotone symplectic
    manifolds admitting semifree circle actions I},
  \newblock preprint, (2018), arXiv:1812.09892v1.

\bibitem{cho2}
  \newblock Y. Cho, 
  \newblock {\em Classification of six dimensional monotone symplectic
    manifolds admitting semifree circle actions II},
  \newblock preprint, (2019), arXiv:1904.10962v1.

\bibitem{cho3}
  \newblock Y. Cho,
  \newblock {\em Classification of six dimensional monotone symplectic
    manifolds admitting semifree circle actions III},
  \newblock preprint, (2019), arXiv:1905.07292v1.

\bibitem{dh}
  \newblock J.J. Duistermaat, G.J. Heckman,
  \newblock {\em On the variation in the cohomology of the symplectic form of
    the reduced phase space},
  \newblock Invent. Math., {\bf 69}, no. 2, (1982), 259 -- 268.

\bibitem{fp_hyp}
  \newblock J. Fine, D. Panov,
  \newblock {\em Hyperbolic geometry and non-K\"ahler manifolds with
    trivial canonical bundle},
  \newblock Geom. Top., {\bf 14}, no. 3, (2010), 1723 -- 1763.

\bibitem{fp}
  \newblock J. Fine, D. Panov,
  \newblock {\em Circle invariant fat bundles and symplectic Fano
    6-manifolds},
  \newblock J. London Math. Soc., {\bf 91}, no. 3, (2015), 709 -- 730.

\bibitem{gvhs}
  \newblock L. Godinho, F. von Heymann, S. Sabatini,
  \newblock {\em 12, 24 and beyond},
  \newblock Adv. Math., {\bf 319}, (2017), 472 -- 521.

\bibitem{gromov}
\newblock M. Gromov, 
\newblock {\em Pseudo holomorphic curves in symplectic manifolds}, 
\newblock Invent. Math., {\bf 82} (1985), 307--347.

\bibitem{gs}
  \newblock V. Guillemin, S. Sternberg,
  \newblock {\em Convexity properties of the moment mapping},
  \newblock Invent. Math., {\bf 67}, no. 3, (1982), 491 -- 513.

\bibitem{gs-local}
  \newblock V. Guillemin, S. Sternberg,
  \newblock {\em A normal form for the moment map},
  \newblock Differential geometric methods in mathematical physics
  ({J}erusalem, 1982), Math. Phys. Stud., {\bf 6}, Reidel, Dordrecht,
  1984, 161 -- 175.

\bibitem{hirze}
  \newblock F. Hirzebruch, T. Berger, R. Jung,
  \newblock {\em Manifolds and modular forms},
  \newblock Aspects of Mathematics, E20, With appendices by Nils-Peter
  Skoruppa and by Paul Baum, Friedr. Vieweg \& Sohn, Braunschweig,
  1992, xii + 211.

\bibitem{isko_prok}
  \newblock V.A. Iskovskikh, Yu. G. Prokhorov,
  \newblock {\em Fano varieties},
  \newblock in Algebraic Geometry, V, Encyclopaedia Math. Sci., {\bf
    47}, Springer, Berlin, 1999, 1 -- 247.

\bibitem{karshon}
  \newblock Y. Karshon,
  \newblock {\em Periodic Hamiltonian flows on four dimensional
    manifolds},
  \newblock Mem. Amer. Math. Soc., {\bf 141}, no. 672, (1999), viii +
  71.

\bibitem{kt1}
  \newblock Y. Karshon, S. Tolman,
  \newblock {\em Centered complexity one {H}amiltonian torus actions},
  \newblock Trans. Amer. Math. Soc., {\bf 353}, no. 12, (2001), 4831
  -- 4861.
 
\bibitem{kt2}
  \newblock Y. Karshon, S. Tolman,
  \newblock {\em Complete invariants for {H}amiltonian torus actions with two
    dimensional quotients},
  \newblock J. Symplectic Geom., {\bf 2}, no. 1, (2003), 25 -- 82.

\bibitem{kt3}
  \newblock Y. Karshon, S. Tolman,
  \newblock {\em Classification of {H}amiltonian torus actions with
    two-dimensional quotients},
  \newblock Geom. Topol., {\bf 18}, no. 2, (2014), 669 -- 716.

\bibitem{kirwan}
  \newblock F.C. Kirwan,
  \newblock {\em Cohomology of quotients in symplectic and algebraic
    geometry},
  \newblock Mathematical Notes, {\bf 31}, Princeton University Press,
  Princeton, NJ, 1984, i+211.

\bibitem{li}
  \newblock H. Li,
  \newblock {\em The fundamental group of symplectic manifolds with Hamiltonian Lie group actions},
  \newblock J. Symplectic Geom., {\bf 4}, no. 3, (2006), 345 -- 372.

\bibitem{lp}
  \newblock N. Lindsay, D. Panov,
  \newblock {\em $S^1$-invariant symplectic hypersurfaces in dimension
    6 and the Fano condition},
  \newblock to appear in J. Top., DOI:10.1112/topo.12087, (2018).

\bibitem{lindsay}
  \newblock N. Lindsay,
  \newblock {\em Hamiltonian circle actions on symplectic Fano
    manifolds},
  \newblock Ph.D. thesis, King's College London, 2018.
      
\bibitem{marle}
  \newblock C.-M. Marle,
  \newblock {\em Mod\`ele d'action hamiltonienne d'un groupe de {L}ie sur une
    vari\'{e}t\'{e} symplectique},
  \newblock Rend. Sem. Mat. Univ. Politec. Torino, {\bf 43}, no. 2,
  (1985), 227 -- 251 (1986).

\bibitem{mcduff_six}
  \newblock D. McDuff,
  \newblock {\em Some $6$-dimensional Hamiltonian $S^1$-manifolds},
  \newblock J. Topol., {\bf 2}, no. 3, (2009), 589 -- 623.

\bibitem{mcduff_structure}
\newblock D. McDuff, 
\newblock{\em The structure of rational and ruled symplectic 4-manifolds}, 
\newblock J. Amer. Math. Soc. {\bf 3}, (1990), 679--712.

\bibitem{mcduff}
  \newblock D. McDuff, D. Salamon,
  \newblock {\em Introduction to symplectic topology},
  \newblock Oxford Mathematical Monographs, Second Edition, The
  Clarendon Press, Oxford University Press, New York, 1998, x+486.

\bibitem{mcduff_sal}
  \newblock D. McDuff, D. Salamon,
  \newblock {\em $J$-holomorphic curves and symplectic topology},
  \newblock AMS Colloquium Publications, {\bf 52}, American
  Mathematical Society, Providence, RI, 2004, xii+669.

\bibitem{rez}
  \newblock A.G. Reznikov,
  \newblock {\em Symplectic twistor spaces},
  \newblock Ann. Global Ann. Geom., {\bf 11}, no. 2, (1993), 109 -- 118.

\bibitem{taubes}
\newblock C.H. Taubes, 
\newblock {\em Seiberg-Witten and Gromov invariants for symplectic 4-manifolds}
\newblock International Press,
Somerville, 2000. 

\bibitem{thurston}
  \newblock W. Thurston,
  \newblock {\em Some simple examples of symplectic manifolds},
  \newblock Proc. Amer. Math. Soc., {\bf 55}, (1976), 467 -- 468.

\bibitem{tolman}
  \newblock S. Tolman,
  \newblock {\em A symplectic generalization of Petrie's conjecture},
  \newblock Trans. Amer. Math. Soc., {\bf 362}, no. 8, (2010), 3963 -- 3996.

\end{thebibliography}
\end{document}